\documentclass[]{amsart}

\usepackage{amssymb}
\usepackage{amsmath}
\usepackage{amsthm}
\usepackage{amscd}

\usepackage{color}

\newtheorem{theorem}{Theorem}

\newtheorem*{lemma}{Lemma}

\def\Z{{\mathbb Z}}

\definecolor{darkgreen}{cmyk}{1,0,1,.2}

\renewcommand\paragraph[1]{\medskip\textbf{#1} }

\begin{document}
\title[Free products]{Commensurability and
  QI classification of free products of
  finitely generated abelian groups}

\author{Jason A. Behrstock} \address{Department of
Mathematics\\Columbia University\\New York, NY 10027}
\email{jason@math.columbia.edu}

\author{Tadeusz Januszkiewicz} \address{Department of
Mathematics\\Ohio State University\\Columbus, OH 43210\\
and Mathematical Instutute of Polish Academy of Sciences;
on leave from Wroclaw University}
\email{tjan@math.ohio-state.edu}

\author{Walter D. Neumann} \address{Department of  
Mathematics\\Barnard College, Columbia University\\New York, NY  
10027} 
\email{neumann@math.columbia.edu}

\thanks{Research supported under NSF grants no.\ DMS-0604524,
DMS-0706259, and DMS-0456227}

\date{\today}

\maketitle

The following gives the complete commensurability and quasi-isometry
classification of free products of finitely generated abelian groups.
The quasi-isometry classification is a special case of  
Papasoglu and Whyte \cite{PapasogluWhyte:ends}.
\begin{theorem}
  Let $G_i$ be a free product of a finite set $S_i$ of finitely generated
  abelian groups for $i=1,2$. Then the
  following are equivalent
    \begin{enumerate}
        \item The sets of ranks $\ge2$ of groups in $S_1$ and $S_2$
          are equal (the \emph{rank} of a finitely generated group is
          the rank of its free abelian part);
	\item $G_{1}$ and $G_{2}$ are commensurable.
	\item $G_{1}$ and $G_{2}$ are quasi-isometric.
    \end{enumerate}
\end{theorem}
\begin{proof} The main step is to show (1) implies (2). By going to
  finite index subgroups of $G_1$ and $G_2$ we can assume the groups
  in $S_1$ and $S_2$ are free abelian (take the kernel of the map of
  $G_i$ to a product of finite quotients of the groups in $S_i$ by
  torsion free normal subgroups, or see the lemma below for a more
  general statement). Let $n_1=1$ and let $n_2,\dots, n_k$ be the
  ranks $\ge 2$ of the groups in $S_1$ and in $S_2$. Let $r_i$ and
  $s_i$ be the number of rank $n_i$ groups in $S_1$ and $S_2$
  respectively.  We identify $G_{1}$ with the fundamental group of the
  topological space $W_{1}$ consisting of the wedge of $r_{i}$
  $n_i$--dimensional tori for each $i$; similarly we let
  $G_{2}=\pi_{1}(W_{2})$, where $W_{2}$ is defined similarly using the
  $s_{i}$'s.  A finite cover of such a wedge of tori is homotopy
  equivalent to a wedge of tori.

  We proceed in two steps. 
  First, using finite covers we
  replace $(r_{1},r_{2},\ldots,r_k)$ and
  $(s_{1},s_{2},\ldots,s_k)$ by the sequences $(R_{1}, Y, Y,
  \ldots,Y)$ and $(S_{1}, Y, Y, \ldots,Y)$, respectively, where
  $R_{1}, S_{1},$ and $Y$ are positive integers.  Then we show
  that, again taking finite covers, we can leave the $Y$'s unchanged and
  replace both $R_{1}$ and $S_{1}$ by a positive integer $X$,
  making the two sequences equal and completing the argument.
    
    For the first step, let $Y$ be a common multiple of 
    $r_{2},\ldots,r_{n},s_{2},\ldots s_{n}$. 
    We construct $W'_1$ and
    $W'_2$ in the following way.  $W'_1$ is a finite cover of $W_1$
    which for each $2\leq i\leq k$
    satisfies the following: it has $Y$ tori of dimension
    $n_i$; each  $n_i$--dimensional torus in 
    $W_1$ has $Y/r_{i}$  $n_i$--dimensional tori in $W'_1$ which project 
    to it; and each torus of $W'_1$ covers its image in $W_1$ with 
    degree $r_i$. Hence  $W'_{1}$ is a covering of 
    $W_{1}$ with degree $Y$.  
    Similarly construct $W'_2$ from $W_{2}$ using the $s_{i}$. 
    Note that $W'_1$ and $W'_2$ are each homotopy equivalent to wedges
    of tori (by contracting an embedded tree connecting the lifts of the
    basepoint), but this construction doesn't control 
    the number of $1$--dimensional tori in these wedges of tori.
    
    For step 2 we notice that, given two spaces which are homotopy
    equivalent to wedges
    of tori of dimension up to $n$, if the number of $n_i$--dimensional
    tori is the same in each for all $2\leq i \leq k$, then the number
    of $1$--tori in the equivalent wedges of tori are equal 
    if and only if the spaces 
    have the same Euler characteristic.  Now, note that a
    cyclic cover of degree $d$ which is a connected cover on each $n_i$--torus
    for $i\ge 2$
    leaves the number of tori of dimensions $\ge 2$ unchanged.  Thus,
    taking $W'_1$ and $W'_2$ as above, we can take a $\chi(W'_2)$--fold
    cyclic cover of $W'_1$ and a $\chi(W'_1)$--fold cyclic cover of
    $W'_2$ to obtain a pair of spaces homotopy equivalent to wedges of
    tori with the same number of tori for
    $2\leq i\leq n$ and the same Euler characteristic. Hence these two
    covers are homotopy equivalent, thereby
    completing the argument.
    
    (2) implies (3) for any finitely generated group.  That (3)
    implies (1) follows from the easy fact that $\Z^{n}$ is
    quasi-isometric to $\Z^{m}$ if and only if $n=m$ plus the result
    of Papasoglu and Whyte \cite{PapasogluWhyte:ends}: given two
    groups which are not virtually $\Z$ and which admit graph of groups
    decompositions whose edge groups are finite and whose vertex groups
    have one or less ends, they are quasi-isometric if and
    only if they both have the same set of quasi-isometry types
    of vertex groups.
\end{proof}

The class of groups considered here is also quasi-isometrically rigid:
\begin{theorem}
  If a group $G$ is quasi-isometric to a free product of finitely
  generated abelian groups then it is commensurable with such a free product.
\end{theorem}
\begin{proof}
  By the main theorem of \cite{PapasogluWhyte:ends} (Theorem 0.4), we
  have that $G$ has a graph of groups decomposition with finite edge
  groups and where each vertex group is quasi-isometric to a finitely
  generated abelian group.  Since a group is quasi-isometric to
  $\Z^{i}$ if and only if it has a finite-index subgroup isomorphic to
  $\Z^{i}$ \cite{Bass:DegreeOfGrowth, Gersten:dimension,
  Gromov:PolynomialGrowth}, we claim that $G$ has a finite-index
  subgroup which is a free product of 
  finitely generated free abelian groups.  This
  follows from the lemma below.
\end{proof}
\begin{lemma}
  If a group $G$ splits as a finite graph of groups with finite edge
  groups and virtually torsion-free vertex groups $G_v$, then $G$ has
  a finite-index subgroup which is a free product of a collection of
  torsion-free subgroups of each of the $G_v$'s together with an
  additional free group.
\end{lemma}
\begin{proof}
  Let $K_e$ denote the edge group corresponding to an edge $e$.  A
  standard classifying space $BG$ for $G$ is obtained by gluing spaces
  into a ``graph of spaces'' with spaces $BG_v$ at the vertices glued
  in the obvious way to spaces $BK_e\times[0,1]$ for the edges.  For
  each $i$ let $N_v$ be a torsion free normal subgoup of $G_v$ of
  index $y_v$ and let $Y$ be a common multiple of the $y_v$. If $e$ is
  an edge that abuts vertex $i$ then the image of $BK_e$ in $BG_v$
  lifts to $y_v/|K_e|$ copies of (the contractible space) 
  $\widetilde{BK}_e$ in the
  cover $BN_v$ of $BG_v$. Take a new graph of spaces where each $BG_v$
  is replaced by $Y/y_v$ copies of $BN_v$ and each $BK_e$ is replaced
  by $Y/|K_e|$ copies of $\widetilde{BK}_e\times[0,1]$. If $e$ abuts
  vertex $i$ we glue $y_v/|K_e|$ copies of the edge space to each of
  the $Y/y_v$ $BN_v$'s. We can clearly do this to get a space which is
  a connected finite cover of the original graph of spaces. It then
  has fundamental group a graph of groups with trivial edge groups and
  with vertex groups equal to copies of the $N_v$. This is a free
  product of the vertex groups and an additional free group given by
  the fundamental group of the underlying graph.
\end{proof}

We note that versions of the results will hold for free products
allowing other groups that have many isomorphic subgroups of finite
index and that are quasi-isometrically rigid up to commensurability,
for example groups of the form (free)$\times \mathbb Z^n$ with $n\ge 1$, the
Heisenberg group, etc.

\end{document}